\documentclass[12pt,a4paper]{amsart}

\usepackage{amsgen, amsmath, amstext, amsbsy, amsopn, amsfonts, amsthm, amssymb, amscd, hyperref}

\newtheorem{thm}{Theorem}
\newtheorem*{thm*}{Theorem}
\newtheorem{lem}{Lemma}

\newtheorem{cor}[thm]{Corollary}

\theoremstyle{definition}
\newtheorem{defn}{Definition}

\theoremstyle{remark}

\DeclareMathOperator{\cat}{cat}

\newcommand{\dsum}{\mathop{\displaystyle \sum}\limits}

\begin{document}

\title{Topological transversals to a family of convex sets}

\author{L.~Montejano}

\thanks{The research of Luis~Montejano is supported by CONACYT, 41340}

\email{luis@matem.unam.mx}
\address{Luis~Montejano, Instituto de Matem\'{a}ticas Universidad Nacion\'{a}l Aut\'{o}noma de M\'{e}%
xico, M\'{e}xico D.F. M\'{e}xico, Fax: (52)5556160348}

\author{R.N.~Karasev}

\thanks{The research of Roman~Karasev is supported by the Dynasty Foundation,
the President's of Russian Federation grant MK-113.2010.1, the Russian Foundation for Basic Research grants 10-01-00096 and 10-01-00139}

\email{r\_n\_karasev@mail.ru}
\address{
Roman Karasev, Dept. of Mathematics, Moscow Institute of Physics
and Technology, Institutskiy per. 9, Dolgoprudny, Russia 141700}

\begin{abstract}
Let $\mathcal F$ be a family of compact convex sets in $\mathbb R^d$. We say that $\mathcal F $ has
a \emph{topological $\rho$-transversal of index $(m,k)$} ($\rho<m$, $0<k\leq d-m$) if there are, homologically, as many transversal $m$-planes to $\mathcal F$ as $m$-planes containing a fixed $\rho$-plane in $\mathbb R^{m+k}$.

Clearly, if $\mathcal F$ has a $\rho$-transversal plane, then $\mathcal F$ has a topological $\rho$-transversal of index $(m,k),$ for $\rho<m$ and $k\leq d-m$. The converse is not true in general. 

We prove that for a family $\mathcal F$ of $\rho+k+1$ compact convex sets in $\mathbb R^d$ a topological $\rho$-transversal of index $(m,k)$ implies an ordinary $\rho$-transversal. We use this result, together with the multiplication formulas for Schubert cocycles, the Lusternik-Schnirelmann category of the Grassmannian, and different versions of the colorful Helly theorem by B\'ar\'any and Lov\'asz, to obtain some geometric consequences.
\end{abstract}


\maketitle

\section{Introduction}

Let us make some definitions. By $M(d,m)$ we denote the space of $m$-planes (by \emph{plane} we mean an affine plane) in $\mathbb R^d$. It can be considered as an open subset of the Grassmannian $G(d+1,m+1)$ (see the details in Section~\ref{trans-def-sec}), and is retractible to the Grassmannian $G(d,m)$ of $m$-dimensional linear subspaces of $\mathbb R^d$. 

\begin{defn}
Let $\mathcal F$ be a family of compact convex sets in $\mathbb R^d$. For $0<m<d$ denote by $\mathcal T_m(\mathcal F)$ the subspace of $M(d,m)$ consisting of all $m$-planes transversal to $\mathcal F$, i.e. intersecting every member of $\mathcal F$. A member of $\mathcal T_m(\mathcal F)$ is called an \emph{$m$-transversal} to $\mathcal F$.
\end{defn}

Informally, we shall say that $\mathcal F $ has a \emph{topological $\rho$-transversal of index $(m,k)$} for $\rho<m$, $0<k\leq d-m$, if there are, homologically, as many transversal $m$-planes to $\mathcal F$ as $m$-planes containing a fixed $\rho$-plane in $R^{m+k}$. The formal definition is as follows.

\begin{defn}
For $\rho<m$, $0<k\leq d-m$ the family $\mathcal F$ has a \emph{topological $\rho $-transversal of index $(m,k)$} if the Schubert cocycle $[\underbrace{0,\ldots,0}_{\rho+1},k,\ldots,k]$ is not zero on $\mathcal T_m(\mathcal F)$ (see Sections~\ref{schubert-sec} and \ref{trans-def-sec} for explanations).
\end{defn}

Clearly, if $\mathcal{F}$ has a $\rho$-transversal plane, then $\mathcal F$ has a topological $\rho$-transversal of index $(m,k)$, if $\rho<m$ and $k\leq d-m$. The converse is not true in general.

Still, if the family $\mathcal F$ has limited size, we claim the following.

\begin{thm}
\label{top-trans}
Let $\mathcal F$ be a family of $\rho+k+1$ compact convex sets in $\mathbb R^d$. If $\mathcal F$ has a topological $\rho$-transversal of index $(m,k)$, then it has an ordinary $\rho$-transversal. 
\end{thm}

In the case $k=1$ the following stronger version of Theorem~\ref{top-trans} is true.

\begin{thm}
\label{top-trans-k1}
Let $0\leq\rho<m\leq d-1$. Let 
$$
\mathcal F = \{A_1,\ldots, A_{\rho+2}\}
$$ 
be a family of $\rho+2$ convex sets in $\mathbb R^d$, and let 
$$
\alpha_i\in A_i\quad i=1,\ldots,\rho+2
$$ 
be some points. Suppose there is not a $\rho$-transversal to $\mathcal F$. Then the inclusion
$$
\mathcal T_m(\{\alpha_1,\ldots,\alpha_{\rho+2}\})\subset \mathcal T_m(\mathcal F)
$$
is a homotopy equivalence. In particular, $\mathcal T_m(\mathcal F)$ has the homotopy type of $G(d-\rho-1,m-\rho-1)$, and in the case $m=\rho+1$ the set $\mathcal T_m(\mathcal F)$ is contractible.
\end{thm}

We use these theorems, together with the multiplication formulas for Schubert cocycles, the Lusternik-Schnirelmann category of the Grassmannian, and different versions of the colorful Helly theorem by B\'ar\'any and Lov\'asz, to obtain some geometric consequences in Sections~\ref{col-helly-trans-sec}, \ref{linear-sec}, \ref{col-helly-semi-sec}, \ref{complex-sec}.

Note that a simple fact on the cohomology of Grassmannians, in the Schubert notation 
$$
[\underbrace{1,\ldots,1}_m]^{d-m}=[\underbrace{d-m,\ldots, d-m}_m]\in H^*(G(d,m), Z_2),
$$
has already given useful geometric applications to transversal planes in~\cite{zivvre1990,dol1993,ziv2004}. Several results on transversals, similar to the results of this paper, can be found in~\cite{abmos2002,bm2002,bmo2002,kar2009}. In~\cite{m2010} Theorem~\ref{top-trans} was conjectured and verified in some low-dimensional cases.

\section{Schubert cycles and cocycles}
\label{schubert-sec}

In this paper we use \v Cech homology and cohomology groups with $Z_2$ coefficients, and omit the coefficients in the notation.

Let $G(d,m)$ be the Grassmannian $m(d-m)$-manifold of all $m$-planes through the origin in $\mathbb R^d$. Our main technical tool in this paper is the Schubert calculus. Although we summarize in this section what we need, good references for the homology and cohomology of Grassmannian manifolds are~\cite{ms1974,pont1950,ch1948}.

From now on let $\lambda_1,\ldots,\lambda_m$ be a sequence of integers such that 
$$
0\leq \lambda_1\leq \dots\leq \lambda_m\leq d-m
$$.

\begin{defn}
Denote the following subset of $G(d,m)$
$$
\{\lambda_1,\ldots,\lambda_m\}=\{H\in G(d,m) : \forall j=1,\ldots, m,\ \dim (H\cap \mathbb{R}^{\lambda _{j}+j})\geq j\}.
$$
\end{defn}

For example
$$
\{H\in G(d,m) : \mathbb R^s\subset H\subset \mathbb R^{m+t}\},
$$ 
which is homeomorphic to $G(m+t-s,t-s)$, is also denoted by $\{\underbrace{0,\ldots,0}_s,t,\ldots,t\}$. 
Another example is
$$
\{\underbrace{t,\ldots,t}_{s},\underbrace{d-m\ldots, d-m}_{m-s}\}=\{H\in G(d,m) : \dim (H\cap \mathbb{R}^{t+s})\geq s\}.
$$

It is known that $\{\lambda_1,\ldots,\lambda_m\}$ is a compact subset of $G(d,m)$ of dimension $\lambda=\lambda_1+\dots+\lambda_m,$ which is a closed connected $\lambda$-manifold except possibly for a closed connected subset of codimension at least three. Thus 
$$
H^\lambda(\{\lambda_1,\ldots,\lambda_m\})=Z_2=H_\lambda(\{\lambda_1,\ldots,\lambda_m\}).
$$ 
In fact, $G(d,m)$ has a CW-complex structure in which the open $\lambda$-cells are the following subsets: 
$$ 
\{H\in G(d,m) : \dim (H\cap \mathbb R^{\lambda_j+j})=j, \dim(H\cap \mathbb R^{\lambda_{j}+j-1})=j-1\}.
$$
Thus $\{\lambda_1,\ldots,\lambda_m\}$ is a subcomplex of $G(d,m)$ and if $0\leq \xi_1\leq \dots \leq \xi_m\leq d-m$ and $\{\xi_1,\ldots,\xi_m\}\leq \{\lambda_1,\ldots,\lambda_m\}$ (component-wise) then $\{\xi_1,\ldots,\xi_m\}$ is a subcomplex of $\{\lambda_1,\ldots,\lambda_m\}$. 

\begin{defn}
Let 
$$
(\lambda_1,\ldots,\lambda_m)\in H_\lambda(G(d,m))
$$ 
be the $\lambda$-cycle, which is induced by the inclusion $\{\lambda_1,\ldots,\lambda_m\}\subset G(d,m)$. These cycles are called \emph{Schubert cycles}.
\end{defn}

A canonical basis for $H_\lambda(G(d,m))$ consists of all Schubert cycles $(\xi_1,\ldots,\xi_m)$ such that $0\leq \xi_1\leq\dots\leq \xi_m\leq d-m$ and $\lambda =\xi_1+\dots+\xi_m$.

\begin{defn}
Let us denote by 
$$
[\lambda_1,\dots,\lambda_m]\in H^\lambda(G(d,m))
$$ 
the $\lambda$-cocycle whose value is one for $(\lambda_1,\dots,\lambda_m)$ and zero for any other
Schubert cycle of dimension $\lambda$. This is a \emph{Schubert cocycle}.
\end{defn}

Thus, a canonical basis for $H^\lambda(G(d,m))$ consists of all Schubert cocycles $[\xi_1,\ldots,\xi_m]$ such that $0\leq \xi_1\leq\dots\leq \xi_m\leq d-m$ and $\lambda =\xi_1+\dots+\xi_m$. If 
$$
j:\{\lambda_1,\ldots,\lambda_m\}\rightarrow G(d,m)
$$ 
is the natural inclusion, then $j^*([\xi_1,\ldots,\xi_m])$ is not zero if and only if 
$$
[\xi_1,\ldots,\xi_m]\leq [\lambda_1,\ldots,\lambda_m],
$$ 
i.e. $\xi_i\leq \lambda_i$ for all $i=1,\ldots,m$. The cohomology classes 
$$
w_i = [\underbrace{0,\ldots,0}_{m-i},\underbrace{1,\ldots,1}_i]
$$ 
are the classical \emph{Stiefel-Whitney characteristic classes} of the standard vector bundle over $G(d,m)$. Similarly, the classes
$$
\bar w_i = [\underbrace{0,\ldots,0}_{m-1},i]
$$ 
are called \emph{the dual Stiefel-Whitney characteristic classes}.

The isomorphism 
$$
D: H_\lambda(G(d,m))\rightarrow H^{m(d-m)-\lambda }(G(d,m))
$$ 
given by 
$$
D((\lambda_1,\ldots,\lambda_m))=[d-m-\lambda_m,\ldots,d-m-\lambda_1]
$$ 
is the classical \emph{Poincar\'{e} duality isomorphism}.

By the above, if $X\subset G(d,m)$ is such that $X\cap \{\lambda_1,\ldots,\lambda_m\}=\emptyset$ and $i_X:X\rightarrow G(d,m)$ is the inclusion, then 
\begin{equation*}
i_X^*(D((\lambda_1,\ldots,\lambda_m)))=i_X^*([d-m-\lambda_m,\ldots,d-m-\lambda_1])=0.
\end{equation*}

\section{The spaces of planes and transversals}
\label{trans-def-sec}

We need to make precise definitions on the space of plane transversals. 

Let $M(d,m)$ be the set of all (affine) $m$-planes in $\mathbb R^d$, in particular, $G(d,m)\subset M(d,m)$. We regard $M(d,m)$ as an open subset of $G(d+1,m+1)$, making the following identifications.

Let $z_0\in \mathbb R^{d+1}-\mathbb R^d$ be some point and, without loss of generality, let $G(d+1,m+1)$ be the space of all $(m+1)$-planes in $\mathbb R^{d+1}$ through $z_0$. Let us identify $H\in M(d,m)$ with the unique $(m+1)$-plane $H'\in G(d+1,m+1)$ which contains $H$ and passes through $z_0$. Thus we have
$$
G(d,m)\subset M(d,m)\subset G(d+1,m+1),
$$
where $M(d,m)$ is an open subset of $G(d+1,m+1)$ and $G(d,m)\subset G(d+1,m+1)$ may be regarded as $\{0,d-m,\ldots,d-m\}$, the set of all $(m+1)$-planes in $\mathbb R^{d+1}$ that contain $\mathbb R^1$. In other words, if $j:G(d,m)\rightarrow G(d+1,m+1)$ is the natural inclusion, then $j(\{\lambda_1,\ldots,\lambda_m\})=\{0,\lambda_1,\ldots,\lambda_m\}$. For example, if $0\leq k\leq d-m$, then $\{0,k,\ldots,k\}$ as a subset of $M(d,m)$ is the set of all $m$-planes $H$ through the origin in $\mathbb R^d$ with the property that $H\subset \mathbb R^{m+k}$.

\begin{defn}
Let $A$ be a subset of a topological space $X$, $i:A\rightarrow X$ be the inclusion, and let $\gamma \in H^*(X)$. We say that $\gamma$ \emph{is zero or not zero} on $A$, provided $i^*(\gamma)$ is zero or not zero, respectively, in $H^*(A)$. We write in this case $\gamma|_A=0$ or $\gamma|_A\not=0$ respectively.
\end{defn}

Let us give the details of the definition of a topological transversal. If $\rho<m$ and $0<k\leq d-m$, we say that $[\underbrace{0,\ldots,0}_{\rho+1},k,\ldots,k]$ is not zero on $\mathcal T_m(\mathcal F)$ if
$$
i^*([\underbrace{0,\ldots,0}_{\rho+1},k,\ldots,k])\in H^{(m-\rho)k}(\mathcal T_m(\mathcal F))
$$ 
is not zero, where 
$$
i^*:H^{(m-\rho)k}(G(m+1,d+1))\rightarrow H^{(m-\rho)k}(\mathcal T_m(\mathcal F))
$$
is the cohomology homomorphism induced by the inclusion $\mathcal T_m(\mathcal F)\subset M(d,m)\subset G(d+1,m+1)$. From the definition of the Schubert cycles it is clear that, informally, $[\underbrace{0,\ldots,0}_{\rho+1},k,\ldots,k]$ is not zero on $\mathcal T_m(\mathcal F)$ iff there are homologically as many transversal $m$-planes to $\mathcal F$ as $m$-planes through a fixed $\rho$-plane in $\mathbb R^{m+k}$. 

From the Poincar\'e duality it follows that the topological $\rho$-transversal of index $(m,k)$ implies the following purely geometrical condition: for any affine plane $A$ of dimension $d-k-\rho-1$ (possibly at infinity) there exists an $m$-transversal $L$ to $\mathcal F$, such that $\dim L\cap A\ge m-\rho-1$. 

\section{Proof of Theorem~\ref{top-trans}}

First, let us define a certain characteristic class of a vector bundle. Consider a vector bundle $\eta :E(\eta) \to M$ of dimension $n$ over a compact smooth manifold without boundary. Let us define a characteristic class (in mod $2$ cohomology) of $\eta$ by the following construction. Let $s_1, \ldots, s_l$ be some sections of $\eta$, denote 
$$
z_{l, r} = \{x\in M: \dim\langle s_1(x), \ldots, s_l(x) \rangle \le r\},
$$
here $\langle \ldots \rangle$ denotes the linear span of vectors. It can be easily seen that the $n\times l$ matrices of rank $\le r$ form a submanifold (possibly, with singularities) of the space of all matrices. It follows from the Thom transversality theorem that $z_{l, r}$ is a submanifold (possibly, with singularities) of $M$ for generic sections $s_1,\ldots, s_l$. Let us define the characteristic class $c_{l, r} (\eta)$ as the Poincar\'e dual to $z_{l, r}$. The definition is correct, because the singularities have $\ge 2$ codimension and do not affect the mod $2$ homology and cohomology. The subspaces of rank $\le r$ matrices are widely used in studying the singularities of smooth maps, such matrices correspond to the Porteous-Thom singularities~\cite{port1971}.

Note that the class $c_{l, r}(\eta)$ is functorial. In order to express it in terms of the Schubert cocycles, let us take $M$ to be the Grassmannian $G(N, n)$ and $\eta$ to be its tautological bundle. Let the sections $s_i$ be given by projections of the respective vectors $v_i\in\mathbb R^N$ to the $n$-subspace $L\subset\mathbb R^N$. If the vectors $v_i$ are chosen to be linearly independent, the set $z_{l, r}$ is described as follows
$$
z_{l, r} = \{L\in G(N, n) : \dim L^\perp\cap V\ge l-r\},
$$
where $V$ is the linear hull of $v_1,\ldots, v_l$, or equivalently
$$
z_{l, r} = \{L\in G(N, n) : \dim L\cap V^\perp\ge n - r\}.
$$
Therefore 
$$
\{z_{l, r}\} = \{\underbrace{N-n-l+r,\ldots, N-n-l+r}_{n-r}, \underbrace{N-n,\ldots,N-n}_{r}\}
$$ 
by definition of the Schubert cycle, which is Poincar\'e dual to the Schubert cocycle $[\underbrace{0,\ldots,0}_r,\underbrace{l-r,\ldots, l-r}_{n-r}]$.

In fact all the above reasonings are standard in the singularity theory and can be restated as follows. We consider continuous fiberwise maps $f: \epsilon^l\to \eta$ over $M$, where $\epsilon$ is the trivial one-dimensional bundle. We define the class of singularities for such maps $f$, which is defined by the condition that the rank of the fiber map is $\le r$. Then we find the characteristic class of these singularities using the standard construction over the Grassmannian.

Now we are ready to prove the theorem. Denote $\mathcal T_m$ the set of $m$-transversals to the family 
$$ 
\mathcal F = \{C_1, \ldots, C_{\rho+k+1}\},
$$
it is a subset of $G(d+1, m+1)$, as defined above. Consider the tautological $m+1$-dimensional bundle $\gamma : E(\gamma)\to G(d+1, m+1)$, and take $l=\rho+k+1$ sections $s_i$ of this bundle over $\mathcal T_m$ by selecting continuously a point $s_i(L)\in L\cap C_i$ ($L$ is an $m+1$-dimensional linear space in $\mathcal T_m\subseteq G(d+1, m+1)$). The continuous selection is possible if all $C_i$'s are strictly convex and have nonempty interior (in this case the intersection $L\cap C$ depends continuously on $L$ in the Hausdorff metric), the other cases are reduced to this by $\varepsilon$-approximating $C_i$'s by ``good'' sets, going to the limit $\varepsilon\to 0$, and using the compactness.

Now it suffices to find an element $L\in\mathcal T_m$ such that the vectors $s_i(L)$ span a linear subspace of $L$ of dimension $\le r=\rho+1$. As it was shown in the beginning of the proof, this is guaranteed by the class
$$
c_{l, r}(\gamma|_{\mathcal T_m}) = [\underbrace{0,\ldots,0}_r,\underbrace{l-r,\ldots, l-r}_{m-\rho}]|_{\mathcal T_m} = [\underbrace{0,\ldots,0}_{\rho+1},\underbrace{k,\ldots, k}_{m-\rho}]|_{\mathcal T_m},
$$
which is nonzero by the definition of the topological $\rho$-transversal of index $(m,k)$.

\section{Proof of Theorem~\ref{top-trans-k1}}

Consider 
$$
\widehat{\mathcal T}_m(\mathcal F)=\{(H,a_1,\ldots,a_{\rho+2}) : H\in \mathcal T_m(\mathcal F),\ a_i\in H\cap A_i\},
$$ 
with the two natural projections
$$
\begin{array}{rcl}
& \widehat{\mathcal T}_m(\mathcal F) &  \\ 
\pi_1\swarrow &  & \searrow \pi_2 \\ 
\mathcal T_m(\mathcal F) &  & A_1\times\dots\times A_{\rho +2}.
\end{array}
$$
Observe that $\pi_1$ is a homotopy equivalence because the fiber 
$$
\pi_1^{-1}(H)=\prod_{i=1}^{\rho+2}(H\cap A_i)
$$ 
is contractible for every $H\in \mathcal T_m(\mathcal F)$. 

Suppose there is no $\rho$-transversal to $\mathcal F$. Then each collection of points $(a_1,\ldots,a_{\rho +2})$ with $a_i\in A_i$ determines a unique $(\rho+1)$-plane $L$ in $\mathbb R^d$. Then $\pi_2^{-1}(a_1,\ldots,a_{\rho+2})$ consists of the $m$-planes in $\mathbb R^d$ that contain $L$, which
is homeomorphic to $G(d-\rho-1,m-\rho-1)$. Moreover, it is easy to see that $\pi_2$ is a fiber bundle with fiber $G(d-\rho -1,m-\rho -1)$. Since its base is contractible , then $\widehat{\mathcal T}_m(\mathcal F)$ , and hence $\mathcal T_m(\mathcal F)$ has the homotopy type of $G(d-\rho-1,m-\rho-1)$. 

The inclusion 
$$
\{(H,\alpha_1,\ldots,\alpha_{\rho+2}) : H\in \mathcal T_m(\mathcal F),\ \alpha_i\in H\}\subset \widehat{\mathcal T}_m(\mathcal F)
$$ 
is a homotopy equivalence. Therefore the inclusion 
\begin{multline*}
\mathcal T_m(\{\alpha_1,\ldots,\alpha_{\rho+2}\})=\\
=\pi_1\{(H,\alpha_1,\ldots,\alpha_{\rho+2}) : H\in \mathcal T_m(\mathcal F),\ \alpha_i\in H\}\subset \mathcal T_m(\mathcal F)
\end{multline*}
is also a homotopy equivalence.

\section{Multiplication in the cohomology of $G(d,m)$}
\label{mult-sec}

In order to apply Theorems~\ref{top-trans} and \ref{top-trans-k1} in geometric situations, we need to remind some useful facts on the multiplication in the cohomology of the Grassmannian. The following is the Pieri formula for the multiplication by a dual Stiefel-Whitney class in the cohomology of the Grassmannian~\cite{ch1948,hil1980A}:
$$
[\lambda_1,\ldots,\lambda_m][0,...,0,k]=\dsum [\xi_1,\ldots,\xi_m],
$$
where the summation extends over all combinations $\xi_1,\ldots,\xi_m$ such that

\begin{enumerate}
\item 
$0\leq \xi_1\leq \ldots\leq \xi_m$;
\item 
$\lambda_j\leq \xi_j\leq \lambda_{j+1}$ for all $j$, where we put $\lambda_{m+1}=d-m$;
\item 
$\dsum_{j=1}^{m}\xi_j=k + \dsum_{j=1}^{m}\lambda_j$.
\end{enumerate}

This formula can be applied to the powers $w_1^n = [\underbrace{0,\ldots,0}_{m-1},1]^n$ of the first Stiefel-Whitney class, to give the following result from~\cite{hil1980A,hil1980B}. 

\begin{thm}
\label{sw1-height}
Let $2m\le d$ (if it is not, we consider $G(d, d-m)\sim G(d,m)$ instead and exchange $m$ and $d-m$), and let $2^s$ be the minimal power of two, satisfying $2^s\ge d$. Denote $w_1$ the first Stiefel-Whitney class of the Grassmannian $G(d, m)$.

1) If $m = 1$, then $w_1^{d-1}\not=0$ and $w_1^d=0$;

2) If $m = 2$, then $w_1^{2^s-2}\not=0$ and $w_1^{2^s-1}=0$;

3) If $m > 2$, then in the case $d=2m=2^s$ we have $w_1^{2^{s-1}}\not=0$; and $w_1^{2^s-2}\not=0$ in other cases.

In all cases $w_1^{d-m}\not=0$, and $w_1^{d-m+1}$ may be zero only for $m=1$, or $m=2$ and $d=2^s$.
\end{thm}

Using this theorem, the Lusternik-Schnirelmann category of $G(d,m)$ can be estimated from below by the standard cohomology product length reasoning (maximum nonzero product length in the reduced cohomology). Let us state the explicit result. 

\begin{thm}
\label{grass-ls}
Let $2m\le d$ (if it is not, we consider $G(d, d-m)\sim G(d,m)$ instead and exchange $m$ and $d-m$), and let $w_1^n\not=0\in H^*(G(d,m))$. Then the Lusternik-Schnirelmann category is estimated as follows
$$
\cat G(d,m) \ge \min\{n+2, m(d-m) + 1\}.
$$
In particular, $\cat G(d,m) \ge \min\{d-m+2, m(d-m)+1\}$, and $\cat G(d,m)\ge \min\{d-m+3, m(d-m)+1\}$ if $m\not=1$, and either $m\not=2$ or $d\not=2^s$. Also, the inequality 
$$
\cat G(d,m) \ge \min\{d-m+2, m(d-m) + 1\}
$$
holds without the restriction $2m\le d$.
\end{thm}

\begin{proof}
Consider two cases: $n=m(d-m)$ and $n<m(d-m)$. In the second case find $\xi\in H^*(G(d,m))$ such that 
$$
\dim \xi = m(d-m)-n,\ \dim\xi>0,\ w_1^n\xi\not=0,
$$
such $\xi$ exists by the Poincar\'e duality, because $n<m(d-m)$.

Then we apply the following well-known lemma (which we are also going to use in further proofs) to the nonzero product $w_1^n$ or $w_1^n\xi$.

\begin{lem}
\label{cup-prod}
Let $X$ be a topological space, $A_1,\ldots, A_l$ be its subspaces such that 
$$
\bigcup_{i=1}^l A_l = X.
$$
Let $\alpha_1,\ldots, \alpha_l$ be some cohomology classes such that 
$$
\alpha_1\cdot\dots\cdot\alpha_l\not=0\in H^*(X).
$$
Then for some $i$ the class $\alpha_i$ is nonzero on $A_i$.
\end{lem}

This lemma shows that we need at least $n+1$ null-homotopic subsets to cover $G(d,m)$ in the first case, and at least $n+2$ null-homotopic subsets in the second case. 

\end{proof}

\section{Transversal analogues of the colorful Helly theorem}
\label{col-helly-trans-sec}

In order to state some geometric results we need to make some definitions and remind some known facts.

\begin{defn}
A family $\mathcal F$ is called \emph{intersecting}, if its intersection is nonempty.
\end{defn}

Recall the colorful Helly theorem of B\'ar\'any and Lov\'asz~\cite{ba1982}, see also~\cite{abbfm2009}.

\begin{thm*}[The colorful Helly theorem]
Let $\mathcal F_1, \ldots,\mathcal F_{d+1}$ be families of convex compact sets in $\mathbb R^d$. Suppose that for any system of representatives $\{X_i\in\mathcal F_i\}_{i=1}^{d+1}$ the intersection $\bigcap_{i=1}^{d+1} X_i$ is non-empty. Then for some $i$ the family $\mathcal F_i$ is intersecting.
\end{thm*}

In the sequel we call the partition $\mathcal F = \bigcup_{i=1}^{d+1} \mathcal F_i$ a \emph{painting with $d+1$ colors}. Subfamilies of $\mathcal F$ that have at most one set of each color are called \emph{heterochromatic}. It is natural to ask, what happens if the number of colors is less than $d+1$. Some results of this kind were already established in~\cite[Theorems~21,22,23]{kar2009}. We are going to prove more results in this direction.

\begin{thm}
\label{col-helly-trans}
Let $\mathcal F$ be a family of $(d-m+1)(\rho+k+1)$ compact convex sets in $\mathbb R^d$ 
painted with $d-m+1$ colors with $\rho+k+1$ convex sets of each color. Suppose that every heterochromatic subset of $\mathcal F$ is intersecting. Suppose also that the class
$$
[\underbrace{0,\ldots,0}_\rho, \underbrace{k,\ldots,k}_{m-\rho}]^{d-m+1}
$$
is nonzero on $G(d, m)$.

Then there exists a color and a $\rho$-transversal plane to all convex sets of $\mathcal F$ painted with this color.
\end{thm}

The condition of the nonzero power in the cohomology can be simplified in the following cases:

\begin{itemize}
\item 
$\rho=m-1$. In this case the Pieri formula (see Section~\ref{mult-sec}) shows that the condition holds if $m\ge d-m+1$, and in some other cases. 
\item 
$k=1$. In this case the transposed (in the sense $G(d,m)\sim G(d,d-m)$) Pieri formula shows that the condition may hold if $(m-\rho)(d-m+1)\le m(d-m)$ depending on the coefficients, arising from applying the Pieri formula several times.
\item
$\rho=m-1$ and $k=1$. In this case by Theorem~\ref{sw1-height} we have two cases:

a) $2m\leq d$. If $m>2$, or $m=2$ and $d$ is not a power of two, then theorem holds. It also holds in some of the other cases.

b) $2m>d$. Hence $d-m<m$ and the theorem holds in this case without other restrictions.
\end{itemize}

Let us give a particular example ($d=4,m=3,k=1,\rho=2$) of this theorem: If $\mathcal F$ is a family of $4$ compact, convex, red sets and $4$ compact, convex, blue sets in $\mathbb R^4$, such that every red set intersects every blue set, then there is a color and a $2$-plane transversal to all convex sets of $\mathcal F$ painted with this color.

\begin{proof}[Proof of Theorem~\ref{col-helly-trans}]
For any color $i$ denote $\mathcal F_i$ the subfamily of $\mathcal F$ consisting of all its sets of color $i$.

Consider a linear $m$-subspace $L\subseteq\mathbb R^d$, and its orthogonal complement $L^\perp$. The projections of $\mathcal F$ to $L^\perp$ satisfy the colorful Helly theorem of dimension $d-m$. Hence for some color $i$ there is a point $x\in L^\perp$ such that for every set of $\mathcal F_i$ its projection contains $x$. It means that $L+x$ is an $m$-transversal to $\mathcal F_i$. Let us paint $L$ with color $i$ in this case. Thus the Grassmannian $G(d,m)$ is covered by $d-m+1$ colors $X_1,\ldots, X_{d-m+1}$.

From Lemma~\ref{cup-prod} it follows that the class $[\underbrace{0,\ldots,0}_\rho, \underbrace{k,\ldots,k}_{m-\rho}]$ is nonzero on some $X_i$, and therefore on the corresponding $\mathcal T_m(\mathcal F_i)$ for some $i$. The last claim is true because the natural projection $\mathcal T_m(\mathcal F_i)\to X_i$ has convex preimages of points and therefore induces an isomorphism of \v Cech cohomology. Then we apply Theorem~\ref{top-trans} and obtain a $\rho$-transversal to $\mathcal F_i$.
\end{proof}

Theorem~\ref{col-helly-trans} may be generalized (modulo some cohomology computations) to the case when the transversal dimension $\rho$ and the number $k$ are chosen independently for every color.

\begin{thm}
\label{col-helly-trans2}
Let $\mathcal F$ be a family of compact convex sets in $\mathbb R^d$, painted with $d-m+1$ colors so that color $i$ has $\rho_i+k_i+1$ sets. Suppose that every heterochromatic subset of $\mathcal F$ is intersecting. Suppose also that the product
$$
\prod_{i=1}^{d-m+1}[\underbrace{0,\ldots,0}_{\rho_i}, \underbrace{k_i,\ldots,k_i}_{m-\rho_i}]
$$
is nonzero on $G(d, m)$.

Then there exists a color $i$ and a $\rho_i$-transversal plane to all convex sets of $\mathcal F$ painted with this color.
\end{thm}

Generally, Theorem~\ref{col-helly-trans2} needs some explicit computations with Schubert cocycles. 
We give a particular case of Theorem~\ref{col-helly-trans2}, where the computations are replaced by a simple inequality.

\begin{cor}
\label{col-helly-trans-ineq}
Let $\mathcal F$ be a family of compact convex sets in $\mathbb R^d$, painted with $k+1$ colors so that color $i$ has $\rho_i+k+1$ sets. Suppose that every heterochromatic subset of $\mathcal F$ is intersecting. Suppose also that
$$
\sum_{i=1}^{k+1} \rho_i \ge k(d-k),
$$
or equivalently
$$
|\mathcal F| \ge kd+2k+1.
$$

Then there exists a color $i$ and a $\rho_i$-transversal plane to all convex sets of $\mathcal F$ painted with this color.
\end{cor}

\begin{proof}
Denote $m=d-k$. The Pieri formula in $H^*(G(d,m))$ implies
$$
[\underbrace{0,\ldots,0}_{\rho_i}, \underbrace{k,\ldots, k}_{m-\rho_i}] = [0,\ldots,0, k]^{m-\rho_i},
$$
and 
$$
\prod_{i=1}^{k+1}[\underbrace{0,\ldots,0}_{\rho_i}, \underbrace{k,\ldots,k}_{m-\rho_i}] = [0,\ldots,0, \underbrace{k,\ldots, k}_{\sum_{i=1}^{k+1} (m-\rho_i)}],
$$
which is nonzero iff $\sum_{i=1}^{k+1} (m-\rho_i)\le m$. The last condition is obviously equivalent to the condition of the theorem.
\end{proof}

We also deduce the following result from Theorems~\ref{top-trans-k1} and \ref{grass-ls}.

\begin{thm} 
\label{col-helly-trans-ls}
Let $\mathcal F$ be a family of $n(\rho+2)$ compact ($\rho\ge 1$, $n\ge 2$), convex sets in $\mathbb R^{n+\rho}$, painted with $n$ colors, in which we have $\rho+2$ convex sets of each color. Suppose that every heterochromatic subset of $\mathcal F$ is intersecting. Then there is a color and a $\rho$-transversal plane to all convex sets of $\mathcal F$, painted with this color.
\end{thm}

A particular example of this theorem ($n=3,\rho=1$) is as follows: If $\mathcal F$ is a family of $3$ compact, convex, red sets; $3$ compact convex, blue sets; and $3$ compact, convex, green sets in $\mathbb R^4$ such that every heterochromatic triple is intersecting, then there is a color and a line transversal to all convex sets of $\mathcal F$ painted with this color. Note that Theorem~\ref{col-helly-trans} fails to resolve this case.

\begin{proof}[Proof of Theorem~\ref{col-helly-trans-ls}]
Put $d=n+\rho$. 

The proof proceeds as the proof of Theorem~\ref{col-helly-trans}. We assume the contrary, but instead of obtaining a zero cohomology product in $H^*(G(d,\rho+1))$, we simply note that the sets $X_i$ cannot cover the Grassmannian $G(d,\rho+1)$ by the definition of the Lusternik-Schnirelmann category. Indeed, they are null-homotopic by Theorem~\ref{top-trans-k1}, the inequalities $\rho\ge 1$, $n\ge 2$ imply 
$$
n\le (\rho+1)(n-1)=\dim G(d, \rho+1),
$$
and 
$$
n = d-\rho < \cat G(d, \rho+1)
$$
by Theorem~\ref{grass-ls}.
\end{proof}

In fact, all the above theorems and theorems in Section~\ref{col-helly-semi-sec} can be generalized to families, where each color contains arbitrary number (not necessarily $\rho_i+k_i+1$) sets.

\begin{defn}
Let $\mathcal F$ be a family of subsets of $\mathbb R^d$. We say that $\mathcal F$ has \emph{property $T_m^n$}, if every subfamily $\mathcal G\subseteq \mathcal F$ of size $\le n$ has an $m$-transversal.
\end{defn}

Evidently, every family has property $T_m^{m+1}$, and $T_0^{d+1}$ implies $T_0^\infty$ (the Helly theorem). There are no Helly-type theorems, where $T_m^n$ implies $T_m^\infty$ for $m>0$ without additional assumptions, see~\cite{cgppsw1994}. Now we give an example, where Theorem~\ref{col-helly-trans2} is generalized.

\begin{thm}
\label{col-helly-trans3}
Let $\mathcal F$ be a family of compact convex sets in $\mathbb R^d$, painted with $d-m+1$ colors, so that every color is used at least once. Suppose that every heterochromatic subset of $\mathcal F$ is intersecting. Suppose also that the product
$$
\prod_{i=1}^{d-m+1}[\underbrace{0,\ldots,0}_{\rho_i}, \underbrace{k_i,\ldots,k_i}_{m-\rho_i}]
$$
is nonzero on $G(d, m)$.

Then there exists a color $i$ such that $\mathcal F_i$ has $T_{\rho_i}^{\rho_i+k_i+1}$ property.
\end{thm}

\begin{proof}
Similar to the above proofs, we conclude that there exists $i$ such that 
$$
[\underbrace{0,\ldots,0}_{\rho_i}, \underbrace{k_i,\ldots,k_i}_{m-\rho_i}]|_{\mathcal T_m(\mathcal F_i)}\not=0.
$$
Consider $\mathcal G\subseteq \mathcal F_i$ such that $|\mathcal G|\le \rho_i+k_i+1$. If $|\mathcal G|< \rho_i+k_i+1$, we repeat some element of $\mathcal G$ several times, and assume that $|\mathcal G|= \rho_i+k_i+1$. Now we see that 
$$
\mathcal T_m(\mathcal F_i)\subseteq \mathcal T_m(\mathcal G),
$$
and 
$$
[\underbrace{0,\ldots,0}_{\rho_i}, \underbrace{k_i,\ldots,k_i}_{m-\rho_i}]|_{\mathcal T_m(\mathcal G)}\not=0.
$$
Hence $\mathcal G$ has $\rho_i$-transversal by Theorem~\ref{top-trans}.
\end{proof}

Note that Theorem~\ref{col-helly-trans3} does not follow from Theorem~\ref{col-helly-trans2} directly. Theorem~\ref{col-helly-trans-ls} is generalized in the same manner, the only change in the proof is the following. By the Lusternik-Schnirelmann reasoning we find $i$ such that the inclusion 
$$
\mathcal T_m(\mathcal F_i)\subseteq M(d,m)
$$
is not null-homotopic. Then the inclusion 
$$
T_m(\mathcal G)\subseteq M(d,m)
$$
is not null-homotopic, because the composition of inclusions
$$
\mathcal T_m(\mathcal F_i)\subseteq \mathcal T_m(\mathcal G)\subseteq M(d,m)
$$ 
is not null-homotopic. Hence $\mathcal G$ has a $\rho$-transversal.

\section{Linear maps of simplicial complexes}
\label{linear-sec}

The transversal results of Section~\ref{col-helly-trans-sec} can be restated as existence of plane transversals to certain sets of faces for linear images of simplicial complexes in $\mathbb R^d$. Let us define such a complex. Denote $[n]=\{1,2,\ldots, n\}$

\begin{defn}
Let $\eta =(n_1,\ldots, n_l)$ be a vector of positive integers greater or equal to $l$. Let $L_\eta$ be the simplicial complex with vertices $[n_1]\times\dots\times [n_l]$, and the maximal simplices of the form
$$
[n_1]\times\dots\times[n_{i-1}]\times\{j\}\times[n_{i+1}]\times\dots\times[n_l],
$$
for every $i\in[l]$ and $j\in [n_i]$.
\end{defn}

\begin{thm}
For any linear map $f: L_\eta\to\mathbb R^d$ there exist $i\in [l]$ and a transversal plane of dimension $n_i-l$ to the images of the simplices 
$$
[n_1]\times\dots\times[n_{i-1}]\times\{j\}\times[n_{i+1}]\times\dots\times[n_l],\ j\in [n_i],
$$
under $f$, provided 
$$
\sum_{i=1}^l n_i \ge (l-1)(d+2)+1.
$$
\end{thm}

\begin{proof}
Denote by $\mathcal F_i$ the images of simplices 
$$
[n_1]\times\dots\times[n_{i-1}]\times\{j\}\times[n_{i+1}]\times\dots\times[n_l],\ j\in [n_i].
$$
Note that the conditions of Theorem~\ref{col-helly-trans-ineq} for $\mathcal F = \bigcup_{i=1}^l\mathcal F_i$ are satisfied, if we put $k=l-1$, $\rho_i=n_i-l=n_i-k-1$. The heterochromatic intersection condition is satisfied, because any heterochromatic intersection already contains a vertex of $L_\eta$ by definition.
\end{proof}

\section{A generalization of the colorful Helly theorem and its transversal analogues}
\label{col-helly-semi-sec}

We are going to generalize Theorem~\ref{col-helly-trans2} to the case, when the heterochromatic intersection condition is replaced by a weaker condition.

\begin{defn}
A family $\mathcal F$ of sets with $|\mathcal F|=k$, is \emph{semintersecting} if all except possibly one of its subsets of size $k-1$ are intersecting.
\end{defn}

For example, a family of three sets is semintersecting if one of them intersects the other two. We shall use the following generalization of the colorful Helly theorem, which is interesting itself.

\begin{lem}
\label{col-helly-semi}
Let $\mathcal F$ be a family of compact convex sets in $\mathbb R^d$ painted with $d+2$ colors. Suppose that
every heterochromatic subfamily of $\mathcal F$ of size $d+2$ is semintersecting. Then there is a color and a point in common to all members of $\mathcal F$ with this color.
\end{lem}

\begin{proof}
First of all, it is clear that every heterochromatic subset of $\mathcal F$ of size $d$ is intersecting.
Furthermore, if every heterochromatic subset of $\mathcal F$ of size $d+1$ is intersecting, we are done by the colorful Helly theorem. Thus, there must be $\{A_0,\ldots,A_d\}\subset\mathcal F$, which is a heterochromatic non-intersecting subfamily with the property that for every $i=0,...,d,$ 
$$
\bigcap_{j\neq i}A_j\neq \emptyset.
$$ 
It follows from the Leray theorem on \v{C}ech cohomology that $\cup_{j=0}^d A_j$ has the homology of the sphere $S^{d-1}$. Hence, by the Alexander duality, $\mathbb R^d\setminus\cup_{j=0}^d A_j$ has exactly two components, one of them being bounded. 

Let $v_0$ be any point of the bounded component of $\mathbb R^d\setminus\cup_{j=0}^d A_j$. Remember that there is a color not used in $\{A_0,\ldots,A_d\}$ so we shall prove that $v_0$ lies in every convex set $X\in\mathcal{F}$ with this color. For every $i=0,\ldots,d$ take 
$$
a_i\in \bigcap_{j\neq i}A_j\cap X
$$ 
Note that $\{a_0,\ldots,a_d\}\subset \mathbb R^d$ is in general position, otherwise by Radon's theorem, $\cap _{j=0}^d A_{j}\neq \emptyset$. Let the simplex $\Delta$ be the convex hull of $\{a_0,\ldots,a_d\}$ and note that 
$$
\partial\Delta \subset \cup_{j=0}^d A_j.
$$ 
For every $i=0,\ldots,d,$ let $C_i=A_i\cap \Delta$. Hence, for every $i=1,\ldots,d,$, 
$$
\bigcap_{j\neq i}C_{j}\neq \emptyset\quad\text{but}\quad \bigcap_{j=0}^d C_j=\emptyset.
$$
Similarly to the case of $A_i$'s, $\cup_{j=0}^d C_j\subset \mathbb R^d$ has the homology of $S^{d-1}$, and therefore $\mathbb R^d\setminus\cup_{j=0}^d C_j$ has exactly two components. Thus, the bounded component of $\mathbb R^d\setminus\cup_{j=0}^d A_j$ is $\Delta\setminus\cup_{j=0}^d A_j = \Delta\setminus\cup_{j=0}^d C_j$, and it is contained in the interior of the simplex $\Delta$. In particular, $v_0\in X$.
\end{proof}

The following theorem is deduced from Lemma~\ref{col-helly-semi} in the same way, as Theorem~\ref{col-helly-trans2} is deduced from the colorful Helly theorem.

\begin{thm}
\label{col-helly-semi-trans}
Let $\mathcal F$ be a family of compact convex sets in $\mathbb R^d$ painted with $d-m+2$ colors so that color $i$ has $\rho_i+k_i+1$ convex sets. Suppose that every heterochromatic subset of $\mathcal F$ of size $d-m+2$ is semintersecting. Suppose also that the product
$$
\prod_{i=1}^{d-m+2}[\underbrace{0,\ldots,0}_{\rho_i}, \underbrace{k_i,\ldots,k_i}_{m-\rho_i}]
$$
is nonzero on $G(d, m)$.

Then there is a color $i$ and a $\rho_i$-transversal plane to all convex sets of $\mathcal F$ painted with this color.
\end{thm}

The condition of the nonzero product in the cohomology can be simplified, e.g. in the case $\rho_i=m-1$, $k_i=1$, using Theorem~\ref{sw1-height}. A particular case of this theorem is the following claim: if $\mathcal F$ is a family of $4$ compact, convex, red sets; $4$ compact, convex, blue sets; and $4$ compact, convex, green sets in $\mathbb R^4$, such that every heterochromatic triple is semintersecting, then there is a color and a $2$-plane transversal to all convex sets of this color. Here $d=4, m=3,\rho_i=2,k_i=1$ and we use the equality
$$
[0,0,1]^3=[1,1,1]\not=0\in H^*(G(4,3)).
$$

Similar to Corollary~\ref{col-helly-trans-ineq}, we deduce the following corollary from Theorem~\ref{col-helly-semi-trans} and the Pieri formula
$$
[\underbrace{0,\ldots,0}_{\rho}, \underbrace{d-m,\ldots, d-m}_{m-\rho}] = [0,\ldots,0, d-m]^{m-\rho}
$$
in $H^*(G(d,m))$.

\begin{cor}
Let $\mathcal F$ be a family of compact convex sets in $\mathbb R^d$ painted with $k+2$ colors so that color $i$ has $\rho_i+k+1$ convex sets. Suppose that every heterochromatic subset of $\mathcal F$ of size $k+2$ is semintersecting. Suppose also that
$$
\sum_{i=1}^{k+2} \rho_i \ge (d-k)(k+1),
$$
or equivalently
$$
|\mathcal F|\ge (d+2)(k+1)
$$
Then there is a color $i$ and a $\rho_i$-transversal plane to all convex sets of $\mathcal F$ painted with this color.
\end{cor}

The following theorem is an analogue of Theorem~\ref{col-helly-trans-ls} for semintersecting families.

\begin{thm}
\label{col-helly-semi-ls}
Let $\mathcal F$ be a family of $n(\rho+2)$ compact, convex sets in $\mathbb R^{n+\rho-1}$ painted with $n$ colors, in which we have $\rho+2$ convex sets of each color, $\rho\ge 2$, $n\ge 3$. Suppose that every heterochromatic subset of $\mathcal F$ of size $n$ is semintersecting. Then there is a color and a $\rho$-transversal plane to all convex sets of $\mathcal F$ painted with this color.
\end{thm}

A particular case of this theorem is ($n=3$, $\rho=2$): If $\mathcal F$ is a family of $4$ compact, convex, red sets; $4$ compact, convex, blue sets, and $4$ compact, convex, green sets in $\mathbb R^4$ such that every heterochromatic triple is semintersecting, then there is a color and a $2$-plane transversal to all convex sets of this color.

\begin{proof}[Proof of Theorem~\ref{col-helly-semi-ls}]
The proof of is essentially the proof of Theorem~\ref{col-helly-trans-ls}, but using Theorem~\ref{col-helly-semi} instead of the colorful Helly theorem. 

We consider the Grassmannian $G(n+\rho-1, \rho+1)$ and cover it with the sets $X_i$, corresponding to existence of $\rho+1$-transversals in given direction for $\mathcal F_i$. If there is no $\rho$-transversal for any $\mathcal F_i$, then all the sets $X_i$ are null-homotopic by Theorem~\ref{top-trans-k1}. The inequalities $\rho\ge 2$, $n\ge 3$ imply
$$
n\le (\rho+1)(n-2)=\dim G(n+\rho-1, \rho+1).
$$
If $n-2<\rho+1$ (equivalently $n<\rho+3$), then Theorem~\ref{grass-ls} gives 
$$
\cat G(n+\rho-1, \rho+1)\ge \rho+3
$$
and $n<\cat G(n+\rho-1, \rho+1)$, which is a contradiction. If $n-2\ge \rho+1$, then Theorem~\ref{grass-ls} (its case $3\le m\le \frac{d}{2}$, where $m=\rho+1$) gives
$$
\cat G(n+\rho-1, \rho+1)\ge n+1,
$$
which is a contradiction too.
\end{proof}

\section{The case of $\mathbb C^d$}
\label{complex-sec}

Most of the previous results remain the same if we replace $\mathbb R^d$ by $\mathbb C^d$, and consider the complex Grassmannian $\mathbb CG(d, m)$, spaces $\mathbb CM(d, m)$, $\mathbb C\mathcal T_m(\mathcal F)$, defined in the corresponding manner. The Schubert calculus is valid too, but with integer coefficients ($\mathbb Z$), so we assume integer coefficients in the cohomology in this section. 

The important thing is that the Pieri formula also holds in the complex case, all the coefficients being positive. This fact guarantees a nonzero product much frequently, compared to the $\mathbb R^d$ case. Let us state the corresponding colorful-Helly-type result.

\begin{thm}
\label{col-helly-trans-C}
Let $\mathcal F$ be a family of compact convex sets in $\mathbb C^d$, painted with $2d-2m+1$ colors so that color $i$ has $\rho_i+k_i+1$ sets. Suppose that every heterochromatic subset of $\mathcal F$ is intersecting. Suppose also that the product
$$
\prod_{i=1}^{2d-2m+1}[\underbrace{0,\ldots,0}_{\rho_i}, \underbrace{k_i,\ldots,k_i}_{m-\rho_i}]
$$
is nonzero on $\mathbb CG(d, m)$.

Then there exists a color $i$ and a complex $\rho_i$-transversal plane to all convex sets of $\mathcal F$ painted with this color.
\end{thm}

From the Pieri formula it follows that in the case, when for all $i$ either $k_i=1$, or $k_i=d-m$, or $\rho_i=m-1$, or $\rho_i=0$, the cohomology product is nonzero iff its total dimension is $\leq m(d-m)$, or equivalently
$$
\sum_{i = 1}^{2d-2m+1} k_i(m-\rho_i) \le m(d-m).
$$

\end{document}